\documentclass[10pt]{article}

\usepackage{amsmath}

\numberwithin{equation}{section}
\usepackage{amsfonts}
\usepackage{amsthm}
\usepackage{amssymb}
\usepackage{bbm}
\usepackage{centernot}
\usepackage[x11names]{xcolor}
\usepackage{chemfig}
\usepackage{amsmath}
\usepackage{cite}
\usepackage[nodayofweek]{datetime} 
\usepackage{dsfont}
\usepackage{enumitem}
\usepackage{euscript}
\usepackage{faktor}
\usepackage{hyperref}
\usepackage{float}
\usepackage{geometry}
\usepackage{graphicx}
\usepackage{mathrsfs}
\usepackage{mathtools}
\usepackage{polynom}
\usepackage{stmaryrd}
\usepackage{tikz}
\usepackage{tikz-cd}
\usepackage{wasysym}
\usepackage{xfrac}
\usepackage[x11names]{xcolor}
\usepackage{mathtools}
\usepackage{setspace}
\usepackage{verbatim}

\newif\ifproofread

\DeclarePairedDelimiter\abs{\lvert}{\rvert}
\DeclarePairedDelimiter\norm{\lVert}{\rVert}

\makeatletter
\newcommand*\bigcdot{\mathpalette\bigcdot@{.5}}
\newcommand*\bigcdot@[2]{\mathbin{\vcenter{\hbox{\scalebox{#2}{$\m@th#1\bullet$}}}}}
\makeatother

\usepackage{mathtools}

\AtBeginEnvironment{align}{\setcounter{equation}{0}}

\newtheorem*{Lemma*}{Lemma}
\newtheorem*{Corollary*}{Corollary}
\newtheorem*{theorem*}{Theorem}
\newtheorem*{definition*}{Definition}

\newtheorem{theorem}{Theorem}[section]

\newtheorem{definition}[theorem]{Definition}
\newtheorem{remark}[theorem]{Remark}
\newtheorem{lemma}[theorem]{lemma}

\newtheorem{corollary}[theorem]{corollary}

\newcommand{\thistheoremname}{}
\newtheorem*{genericthm}{\thistheoremname}

\newcommand{\NN}{\mathbb{N}}

\newcommand{\QQ}{\mathbb{Q}}

\newcommand{\RR}{\mathbb{R}}

\newcommand{\ZZ}{\mathbb{Z}}

\newcommand{\eps}{\epsilon}

\newcommand{\pid}{\mathrel{\ooalign{$\lneq$\cr\raise.22ex\hbox{$\lhd$}\cr}}}

\AtBeginDocument{}

\title{A Note on Weyl's Equidistribution Theorem}
\author{Yuval Yifrach\footnote{Technion - Israel Institute of Technology} \footnote{Address all correspondences to yuval@campus.technion.ac.il}}
\Huge
\begin{document}
\maketitle
\begin{abstract}
H. Weyl proved in \cite{Weyl} that integer evaluations of polynomials are equidistributed mod 1 whenever at least one of the non-free coefficients (namely a coefficient of a monomial of degree at least $1$) is irrational. We use Weyl's result to prove a higher dimensional analogue of this fact. Namely, we prove that evaluations of polynomials on lattice points are equidistributed mod 1 whenever at least one of the non-free coefficients is irrational. This result improves the main result of Arhipov-Karacuba-\v{C}ubarikov in \cite{PolWeyl}. We prove this analogue as a corollary of a theorem that guarantees equidistribution of lattice evaluations mod 1 for all functions which satisfy some restraints on their derivatives. Another corollary we prove is that for $p\in(1,\infty)$ the $\ell^p$ norms of integer vectors are equidistributed mod 1.
\end{abstract}
\maketitle

\section{Introduction}
The following definition lies in the heart of all the results in this paper.
\begin{definition}\label{ED}
	Let $n\in \NN$, let $\Gamma\subset \RR^n$ a discrete set and let $F:D\subset \mathbb R^n\rightarrow \mathbb R$. We say that $F(\Gamma)\bmod1$ is equidistributed if the sequence of discrete counting measures $\sigma^D_R$ of $\Gamma$-points in $B_R(0)\cap D$ satisfies:
	\begin{equation}
		\lim_{R\rightarrow \infty}(F\bmod1)_*\sigma^D_R=\lambda
	\end{equation} 
in the weak sense, where $B_R(0)$ denotes the Euclidean ball of radius $R$ around the origin in $\mathbb R^n$, $F\bmod1:D\rightarrow\mathbb R/\mathbb Z$ is the composition of $F$ with the quotient map, $(F\bmod1)_*\sigma^D_R$ denotes the push-forward measure, and $\lambda$ is the Haar probability measure on $S^1$, the unit circle.
\end{definition}
H. Weyl proves in \cite{Weyl} that for any one variable polynomial $f$ with at least one non-free irrational coefficient (namely a coefficient of a monomial of degree at least $1$), $f(\ZZ)\bmod 1$ is equidistributed. This result has various proofs (e.g. Weyl's original proof in \cite{Weyl} and Furstenberg's dynamical proof in \cite[p. 116]{MyFavBook}) and generalisations (e.g. \cite{Bosh} for when instead of polynomials one considers a richer family of functions), including some results for joint distribution of several functions (see \cite{Bosh}). The main result of \cite{PolWeyl} shows that for any multivariable polynomial $F:\RR^n\rightarrow \RR$ with coefficients that satisfy certain Diophantine approximation conditions, $F(\ZZ^n)\bmod 1$ is equidistributed. A more natural (and in fact 'word to word') extension of Weyl's result is the following theorem:
\begin{theorem}\label{probOp}
	Let $n\in \NN$ and let $F:\RR^n\rightarrow \RR$, be a polynomial given by $F(x)=\sum_{\ell}\alpha_{\ell} x^{\ell}$ where $\ell=(\ell_1,\dots,\ell_n)$ is a multi-index. Assume that there exist $\ell$ such that $\abs{\ell}>0$ and $\alpha_{\ell}\notin\QQ$, then $F(\ZZ^n)\bmod1$ is equidistributed. 
\end{theorem}
The method we use to prove Theorem \ref{probOp} is general and is applicable for the proof of the following theorem as well:
\begin{theorem}\label{WeylPar}
	Let $n\in \NN$ and let $M\in \operatorname{GL}_n(\RR)$. Let $S\subset S^{n-1}$ be spherical cap inside the unit Euclidean sphere in $\RR^n$ (namely, a subset of the sphere bounded by the intersection of a hyperplane with the sphere) and let $f:\RR_+\cdot S\rightarrow \RR$ denotes a smooth homogeneous (namely $f(sx)=sf(x)$ for all $x\in \RR_+\cdot S$ and $s\in \RR_+$) function where for a subset $I\subset \RR_+$, $I\cdot S:=\{ta:t\in I,a\in S\}$. Assume there exists $v\in M\ZZ^n$ such that:
	\begin{equation*}
		\left(\frac{\partial f}{\partial v}\bmod1\right)_*m_S\ll \lambda
	\end{equation*}
where $\lambda$ is the Haar probability measure on $S^1$ and $m_S$ is the restriction of the $S^{n-1}$ area measure to $S$. Then for every $u\in \RR^n$, $f(M\ZZ^n+u)\bmod1$ is equidistributed.
	\end{theorem}
\begin{remark}
The most general statement that can be proved along the lines of Theorems \ref{probOp},\ref{WeylPar} is given below as Theorem \ref{WeylGen}. For sake of keeping the introduction concise, we restricted this introduction to the two theorems above. We point out that Theorem \ref{WeylGen} is heavily used in \cite{stat} to prove the main theorem. 
\end{remark}
Already as weaker form of Theorem \ref{WeylGen}, Theorem \ref{WeylPar} can be used to prove the following corollary:
\begin{corollary}\label{normED}
	Let $n\in \NN$ and $1<p<\infty$. Then $\norm{\cdot}_{p}(\ZZ^n)\bmod1$ is equidistributed, where $\norm{\cdot}_p$ is the $\ell_p$-norm on $\RR^n$.
\end{corollary}

\section{Proof of the Main Theorem}
We need the following definition.
\begin{definition}[Directional derivatives]\label{def: part derivative}
Let $U\subset \RR^n$ be open, let $x\in U$ and let $f:U\rightarrow \RR$ be smooth. For any $v\in \RR^n$ and for small enough $\eps>0$ (depending on $v,x,U$), we define the function $f_v:(-\eps,\eps)\rightarrow \RR$ by $f_v(t)=f(x+tv)$. Given $d\in \NN$, the $d$'th derivative of $f$ in direction $v$ is defined as $\frac{\partial^d f}{\partial v^d}:=f_v^{(d)}(0)$.
\end{definition}
We define a family of measures on the unit circle $S^1$. This family is used to state the main result, namely Theorem \ref{WeylGen}. The importance of this family will become clear with the statement of Theorem \ref{WeylGen}.
\begin{definition}\label{irMeas}
	Let $n\in \NN$ and $M\in \operatorname{GL}_n(\RR)$. Let $\mu$ be a probability measure on $S^1$, $m_{S^{n-1}}$ be the Haar probability measure on $S^{n-1}$ and let $S\subset S^{n-1}$ be a spherical cap. Let $f:S\cdot \RR_+\rightarrow \RR$ be a smooth function. We denote $m_S$ to be the restriction of $m_{S^{n-1}}$ to $S$. 
	We say that $f$ is $M-S$-related to $\mu$ if there exists $d\in \NN$ and $v\in M\ZZ^n$ such that:
	\begin{equation}\label{homCond}
		\frac{\partial^{d} f}{\partial v^{d}}\text{ is degree-$0$-homogeneous (i.e. invariant under multiplication by positive scalars)},
	\end{equation}
	\begin{equation}\label{pushCond}
				\left(\frac{\partial^{d} f}{\partial v^{d}}\bmod 1\right)_*m_{S}=\mu
	\end{equation}
	and
	\begin{equation}\label{pushCondHom}
				\frac{\partial^{d+1} f}{\partial v^{d+1}}(x)\rightarrow 0\text{ as $\norm{x}_2\rightarrow \infty$}.
	\end{equation}
	When $S=S^{n-1},M=I$ and the above holds we say simply that $f$ is related to $\mu$.
\end{definition}
\begin{remark}
	The discussion in this section can be carried in a slightly more general context. Namely, we can replace $S^1$ with the $k$-dimensional torus for $k\in \NN$. The arguments and definitions will be identical to the case $k=1$ up to the appropriate modifications in the notation. For simplicity of this note, we chose not to treat this case.
\end{remark}
In the following theorem we state equidistribution results for functions related (in the sense of Definition \ref{irMeas}) to two families of measures.
\begin{theorem}\label{WeylGen}
	Let $n\in \NN$, $M\in \operatorname{GL}_n(\RR)$, $S\subset S^{n-1}$ a spherical cap and recall that $\lambda$ is the Haar probability measure on $S^1$. Let $\mu$ denote a probability measure which is either absolutely continuous w.r.t. $\lambda$ or supported on one irrational point in $S^1$ (namely a point of the form $e^{2\pi i\alpha}$ for $\alpha\in \RR\setminus\QQ$). Then for any smooth $f:S\cdot \RR_+\rightarrow \RR$ which is $S-M$-related to $\mu$ and for any $u\in \RR^n$, $f(M\ZZ^n+u)\bmod 1$ is equidistributed.
\end{theorem}
We united the two claims of the above theorem because their proofs are conceptually the same. As first step towards proving Theorem \ref{WeylGen}, we need to define discrepancy and introduce some notation:
\begin{definition}\label{defDisc}
	Given $N,d\in \NN$, $\eps>0$ and a sequence $\{x_n\}_{n\in \NN}\subset \RR$, we define:
	\begin{enumerate}
		\item The discrepancy of $\{x_n\}_n$ after $N$ steps by:
	\begin{equation*}
		D_N(\{x_n\}_n)=\sup_{I=[a,b]\subset[0,1]}\abs{\mu_N(I)-(b-a)}
	\end{equation*}
	where $\mu_N$ is the uniform measure on the points $\{x_n\bmod1\}_{n=1}^N\subset [0,1]$;
	\item The set:
	\begin{equation*}
		A^d(N,\eps)=\{a\in [0,1]:D_l\left(\{[an^d+P(n)]\bmod1\}_{n\in \NN}\right)>\eps,
	\end{equation*}
	\begin{equation*}
		\text{ for some polynomial }P\text{ of degree at most $d-1$ and every }l=1,\dots,N\}.
		\end{equation*}
	\end{enumerate}
\end{definition}
As a second step we need the following lemma:
\begin{lemma}\label{lemlem}
	For any $d\in \NN$ and $\eps>0$ the following assertions hold:
	\begin{enumerate}
		\item For every $N\in \NN$, $A^d(N+1,\eps)\subset A^d(N,\eps)$;
		\item $\bigcap_{N\geq 1}A^d(N,\eps)\subset \QQ\cap [0,1]$, therefore $\lambda(A^d(N,\eps))\rightarrow 0$ as $N\rightarrow \infty$.
	\end{enumerate}
\end{lemma}
\begin{proof}
	The first assertion is immediate upon noticing that enlargement of $N$ adds conditions to the definition of $A^d(N,\eps)$. The second follows from the celebrated Erd\H{o}s-Turan inequality since the bound on the exponential sum of $(an^d+P(n))_{n\in \NN}$ for irrational $a$ is uniform in $P$ (where $\deg(P)<n$) by the remark following Equation (9) in \cite{Weyl}.
\end{proof}
Now we are ready to prove Theorem \ref{WeylGen}.
\begin{proof}[Proof of Theorem \ref{WeylGen}]
Before we begin we remark that the notation $a$ will usually denote an element of $S^1$, which is sometimes be thought of as $[0,1)$ under the identification $[0,1)\ni t\mapsto e^{2\pi it}\in S^1$. Let $\mu$ be a measure on $S^1$ satisfying the conditions of the theorem, $S\subset S^{n-1}$ be a spherical cap and let $f:S\cdot\RR_+\rightarrow \RR$ be a smooth function which is $S-M$-related to $\mu$ with some $d\in \NN$.
 Let $\eps,\delta>0$, $N_0\in \NN$ to be determined. Define for any $T>0$:
	\begin{equation*}
		B_{T,\eps}:=\{p\in (M\ZZ^n+u)\cap [0,T]S:\text{ }\frac{\partial^d f}{\partial v^d}(p)\bmod1\notin A^d(N_0,\eps)\}
	\end{equation*}
	\begin{equation*}
		=\{p\in (M\ZZ^n+u)\cap [0,T]S:\text{ }\frac{\partial^d f}{\partial v^d}(p/\norm{p}_2)\bmod1\notin A^d(N_0,\eps)\}
	\end{equation*}
	where the equality holds by (\ref{homCond}) (recall Definition \ref{defDisc}). Recall Definition \ref{ED} and denote for all $T>0$:
	\begin{equation*}
			\sigma_T^S=\sigma_T^{S\cdot \RR_+}
	\end{equation*}
	 to be the counting measure of $M\ZZ^n+u$ in $S\cdot \RR_+$.
	 We claim that in both cases: where $\mu$ is supported on an irrational point and where $\mu\ll\lambda$, we can make sure that for some $\eps_k\rightarrow 0$, large enough $N_0=N_0(\eps_k,\delta)$ and $T>T_0(\eps_k,\delta)$:
	 \begin{equation}\label{both}
	 	\sigma_T^S(B_{T,\eps_k}^c)\leq C_2\delta
	 \end{equation}
for some fixed $C_2>0$ depending only on $M$ (recall that $B_{T,\eps}$ depends on $N_0$). Indeed, suppose that $\mu$ is supported on an irrational point, call it $\alpha\in S^1$. In this case, the irrationality of $\alpha$ and Lemma \ref{lemlem}(2) imply that for all $\eps>0$:
\begin{equation}\label{aa}
	 \text{for all large enough }N_0, \alpha\notin A^d(N_0,\eps).
\end{equation}
	Otherwise, suppose that $\mu\ll\lambda$. Then by Lemma \ref{lemlem}(2) and since $\mu\ll\lambda$, for all $\eps>0$:
\begin{equation}\label{b}
	 \text{for all large enough }N_0, \mu(A^d(N_0,\eps))<\delta.
\end{equation}	
Now let $\mu$ satisfy either one of the conditions of Theorem \ref{WeylGen}. Denote $\pi:S\cdot \RR_+\setminus\{0\}\rightarrow S$ to be the projection onto S, namely $x\mapsto x/\norm{x}_2$. It is well known (see \cite{pom}\footnote{This reference deals with the case $M=I$ and $u=0$ and yields $\tilde m_S=m_S$. However, the general case of $M\in \operatorname{GL}_n(\RR)$ and $u\in \RR^n$ follows from this case. See Remark \ref{rem: davenport for grids} for more details.}) that there exists $0<C_1<C_2$ depending on $M$ and a probability measure $\tilde m_{S}$ on $S$ such that
\begin{enumerate}
	\item $\pi_*\sigma^S_T\rightarrow \tilde m_S$ as $T\rightarrow \infty$;
	\item $C_1m_S\leq \tilde m_S\leq C_2m_{S}$.
\end{enumerate}
Therefore for all $\eps>0$:
	\begin{equation}\label{one}
		\sigma_T^S(B_{T,\eps}^c)=\pi_*\sigma_T^S(\{p\in S:\frac{\partial^d f}{\partial v^d}(p)\bmod1\in A^d(N_0,\eps)\})\xrightarrow{T\rightarrow \infty} \tilde m_S(\pi(B_{\infty,\epsilon}^c))
	\end{equation}
	\begin{equation*}
		\leq C_2m_S(\pi(B_{\infty,\epsilon}^c))=\left(\frac{\partial^d f}{\partial v^d}\bmod1\right)_*m_S(A^d(N_0,\eps))\leq \begin{cases}
    0,& \text{when }\mu=\delta_{\alpha}\text{ by equations }(\ref{aa}),(\ref{pushCond}); \\
    C_2\delta,& \text{when }\mu\ll\lambda\text{ by equations }(\ref{b}),(\ref{pushCond}),
\end{cases}
	\end{equation*}
	by the definition of $B_{T,\eps}$, the remark on the weak convergence of $\pi_*\sigma_T^S$. Thus for all $T$ large enough equation (\ref{both}) holds. To use the weak convergence in (\ref{one}), we still need to verify that:
	\begin{equation}
		m_S\left(\partial \left[\left(\frac{\partial^d f}{\partial v^d}\bmod 1\right)^{-1}A^d(N_0,\eps)\right]\right)=0
	\end{equation}
	for an appropriate family of $\eps$'s that converge to $0$. Since $\tilde m_S\leq C_2m_S$, this will show the same for $\tilde m_S$ which is sufficient to apply the weak convergence in (\ref{one}).
	
	For any natural $k$ and polynomial $P$ of degree at most $d-1$, define a function $G_k(P,\cdot):[0,1]\rightarrow \RR_+$ by: 
	\begin{equation*}
		[0,1]\ni a\mapsto D_k\left(\left\{([an^d+P(n)]\bmod1\right\}_{n\in \NN}\right)\in \RR_+.
	\end{equation*}
Denote $\mu_{a}$ to be the uniform measure on $\{[a1^d+P(1)]1,\dots,[ak^d+P(k)]\bmod1\}\subset S^1$ for any $a\in [0,1)$.
Fix $a\in[0,1)$ and let $[0,1)\ni a_n\rightarrow a$. Let $I\subset [0,1)$ be an interval and let $m$ be large enough such that $\mu_{a_n}(I)\leq\mu_{a}(I+[-1/m,1/m])$ for every $n\geq m$. Such $m$ can be found since we only consider finitely many points. Therefore:
	\begin{equation*}
		\sup_{I\subset S^1}\abs{\lambda(I)-\mu_{a}(I)}\geq \sup_{I\subset S^1}\abs{\lambda(I)-\mu_{a_n}(I)}-\frac{2}{m}.
	\end{equation*}
Since this inequality is symmetric (we can replace $\theta,\theta_n$), we get continuity of $G_k(P,\cdot)$. By similar arguments it is readily seen that $G_k$ is also continuous in the coefficients of $P$. Note that by definition of $D_k$ (see Definition \ref{defDisc}) the function $G_k(P,\cdot)$ only depends on the values of the coefficients of $P$ modulo $1$ so in fact:
\begin{equation*}
	G_k(a):=\sup_{\overline b=(b_1,\dots,b_d)\in [0,1]^d}G_k(P_{\overline b},a)=\sup_{P \text{ Polynomial of degree $\leq d-1$}}G_k(P,a)
\end{equation*}
where for $\overline b=(b_1,\dots, b_d)$ we denote $P_{\overline b}(t)=b_1t^{d-1}+\dots+b_d$. Therefore $G_k(\cdot)$ is continuous as the supremum on the first parameter of $\tilde{G_k}:[0,1]^d\times S^1\rightarrow \RR_+$, where $\tilde{G_k}(\overline b,a):=G_k(P_{\overline b},a)$ (it holds since the domain of $\tilde{G_k}$ is compact).
Note that by definition of $A^d(N,\eps)$ and of $G_k$, $A^d(N_0,\eps)=\{a\in [0,1):G_k(a)>\eps\text{ for every }k=1,\dots,N_0\}$ therefore by continuity of $G_k(\cdot)$:
\begin{equation}\label{myEq}
	\partial A^d(N_0,\eps)=\{a\in [0,1):G_k(a)=\eps\text{ for some }k=1,\dots,N_0\}.
\end{equation}
Note that for every $k$, the set $C_k=\{t\in \RR_+:\mu(\{G_k=t\})>0\}$ must be countable (otherwise we would find uncountably many disjoint subsets of $S^1$ with positive $\mu$ measure which contradicts the assumption on $\mu$), so $C=\bigcup_{k\geq 0}C_k$ is countable as well. In particular, we can take $\epsilon_k\rightarrow 0$ such that $\epsilon_k\notin C$ for every $k$. Note that by (\ref{myEq}), for every $N$, $\partial A(N,\epsilon)\subset \bigcup_{M\leq N}\{G_M=\eps\}$ and if $\eps\notin C$ we get in particular that $\mu(\partial A(N,\epsilon))=0$. Recall the following general fact: for any continuous function $g:X\rightarrow Y$ between metric spaces, and a subset $A\subset Y$ it holds that $\partial(g^{-1}A)\subset g^{-1}(\partial A)$. Therefore by (\ref{pushCond}):
\begin{equation*}
	m_S\left(\partial\left[\left(\frac{\partial^d f}{\partial v^d}\bmod 1\right)^{-1}A(N,\eps_k)\right]\right)\leq m_S\left(\left(\frac{\partial^d f}{\partial v^d}\bmod 1\right)^{-1}\partial A(N,\eps_k)\right)=\mu(\partial A(N,\eps_k))=0
\end{equation*}
so we justified the use of weak convergence in equation (\ref{one}) with the sequence $(\eps_k)_k$.\\
For every $\eps>0$ and $p\in B_{\infty,\eps}$, denote:
\begin{enumerate}
	\item $I_p=\{p+kv:k=1,\dots,N_p\}$;
	\item $f_p(t)=f(p+tv);t\in [0,N_p]$;
	\item $P_p(t)=\frac{\partial^d f}{\partial v^d}(p)t^d+\dots+f(p)$,
\end{enumerate}
where $N_p$ is an integer in $[1,N_0]$ for which $G_{N_p}(\frac{\partial^d f}{\partial v^d}(p))<\epsilon$ (which exists since $p\in B_{\infty,\eps}$). By (\ref{pushCondHom}), for $T>T_0$ large enough $\abs{\frac{\partial^{d+1}f}{\partial v^{d+1}}(p)}<\epsilon/N_0^d$ when $p\in [T,\infty)S$. By Taylor's Remainder theorem:
\begin{equation}\label{three}
\norm{f_p-P_p}_{\infty}\leq C\epsilon	
\end{equation}
for any $p\in [T,\infty)S$, after possibly modifying $T_0$ and for some absolute constant $C$ (not depending on $N_0$). By (\ref{three}) and Definition \ref{defDisc}, for any $p\in B_{T,\eps}$ ($T>T_0$):
\begin{equation}\label{eq}
	D_{N(p)}(\{f_p(k)\}_{k\in \NN})\leq C\epsilon
\end{equation}
for some absolute constant $C$. By equation (\ref{both}) we may find for any $k\in \NN,\delta>0$ positive numbers $T>0,N_0\in \NN$ large enough, and points $\{p_i^T\}_{i=1}^{N_T}\subset B_{T,\epsilon_k}$ for some $N_T\in \NN$, such that:
\begin{equation}\label{blahh}
 \frac{\abs{\bigsqcup_{i=1}^{N_T}I_{p_i^T}}}{\abs{[0,T]S\cap (M\ZZ^n+u)}}\geq (1-C_2\delta)+o_T(1);
\end{equation}
\begin{equation}\label{blah}
	 I_{p_i^T}\subset [0,T]S\cap (M\ZZ^n+u). 
\end{equation}
Note that the above two geometric properties follow from Equation (\ref{both}) and the fact that $v\in M\ZZ^n$. Now by equations (\ref{eq}),(\ref{blahh}),(\ref{blah}) we deduce that for any $k\in \NN, I\subset S^1$ and $T$ large enough:
\begin{equation*}
	\abs{\sigma_T^S\left((f\bmod1)^{-1}I\right)-\lambda(I)}\leq C_2\delta+(1-C_2\delta+o_T(1))C\eps_k
\end{equation*}
which implies the claim of the theorem, by taking $T\rightarrow \infty$ and then taking $k\rightarrow \infty$, $\delta\rightarrow 0$ and using the fact that $\eps_k\rightarrow 0$.
\end{proof}
\begin{remark}
	Although in Definition \ref{irMeas} $S$ is a subset of $S^{n-1}$, a larger family of sets can replace it in Theorem \ref{WeylGen}. For example, any open bounded subset $S$ of a level surface of a homogenous function, if $S$ is Jordan measurable in the level surface. Since the proof is identical we chose to treat the case of the sphere which is easier to visualise.
\end{remark}
\begin{remark}\label{rem: davenport for grids}
	In the proof, we required the following fact. Let $M\in\operatorname{GL}_n(\RR)$ and let $u\in \RR^n$. For simplicity, we deal with the case $S=S^{n-1}$. The case of general $S$ is similar. Denote for any $R>0$, by $\pi_*\sigma^{M,u}_R$ the measure on $S^{n-1}$ coming from projecting the points $M\ZZ^n+u\cap B_R(0)$ to $S^{n-1}$ with the map $\pi$. Then $\pi_*\sigma_R\rightarrow \tilde m_{S^{n-1}}$ weakly as $R\rightarrow \infty$ where $C_1m_{S^{n-1}}\leq \tilde m_{S^{n-1}}\leq C_2m_{S^{n-1}}$. The reference \cite{pom} shows this claim for $M=I$ and $u=0$. We sketch a proof of how to deduce the general case from this. Let $A\subset S^{n-1}$ be open and Jordan measurable. We know from the case $M=I,u=0$ that $\pi_*\sigma^{I,0}_R(\RR_+\cdot M^{-1}A\cap S^{n-1})\rightarrow m_{S^{n-1}}(\RR_+\cdot M^{-1}A\cap S^{n-1})$ as $R\rightarrow \infty$. Moreover, one readily checks that there exist $0<C_1<C_2$ depending on the operator norm of $M$ such that $C_1 m_{S^{n-1}}(A)\leq m_{S^{n-1}}(\RR_+\cdot M^{-1}A\cap S^{n-1})\leq C_2m_{S^{n-1}}(A)$. Define $\tilde m_{S^{n-1}}(A)=m_{S^{n-1}}(\RR_+\cdot M^{-1}A\cap S^{n-1})$. Next, note that for any $u\in \RR^n$ and $M\in \operatorname{GL}_n(\RR)$ there exists $C=C(M)>0$ such that $\pi_*\sigma_R^{M,u}(A)=\pi_*\sigma_{C(M)R}^{I,0}(\RR_+\cdot M^{-1}A\cap S^{n-1})+o_R(1)$. Taking $R\rightarrow\infty$ shows $\pi_*\sigma_R^{M,u}(A)\rightarrow \tilde m_{S^{n-1}}(A)$ and we already know that $\tilde m_{S^{n-1}}$ satisfies $C_1m_{S^{n-1}}\leq \tilde m_{S^{n-1}}\leq C_2m_{S^{n-1}}$.
\end{remark}

\section{Proof of Theorem \ref{probOp} and Corollary \ref{normED}}\label{OpSec}
We start by proving Theorem \ref{probOp}. Our method is finding, for any polynomial $F$ as in Theorem \ref{probOp}, $v\in \ZZ^{n}$ such that $F$ and $v$ satisfy equations (\ref{homCond})-(\ref{pushCondHom}) for $d=\operatorname{deg}(F)$ and a measure $\mu$ on $S^1$ which is supported on an irrational point. In the language of Definition \ref{irMeas}, we will find a Dirac measure $\mu$ supported on an irrational point such that $F$ is related to $\mu$. This will prove Theorem \ref{probOp} using Theorem \ref{WeylGen}. First, we prove the existence of such $v$ under an extra assumption:
\begin{lemma}\label{a}
	Let $m,n,d\in \NN$ and let $F:\RR^n\rightarrow \RR$ be $F(x)=\sum_{\ell}\alpha_{\ell}x^{\ell}$ where $\ell=(\ell_1,\dots,\ell_n)$ is a multi-index and $d=\max\{\abs{\ell}:\alpha_{\ell}\neq 0\}$ is the degree of $F$. Assume that there exist $\ell$ such that $\abs{\ell}=d$ and $\alpha_{\ell}\notin\QQ$. Then there exists $v\in \ZZ^n$ and an irrational point $\alpha_0\in (0,1)$ such that:
	\begin{equation*}
			\frac{\partial^{d} F}{\partial v^{d}}\equiv \alpha_0.
	\end{equation*}
In other words, $F$ is related to $\delta_{\alpha_0}$.
\end{lemma}
\begin{proof}
Assume for the sake of contradiction that:
\begin{equation}\label{ratEq}
	\forall v\in \ZZ^n:\frac{\partial^d F}{\partial v^d}\in \QQ.
\end{equation}
Let $A_{n,d}=\{\ell:\abs{\ell}=d\}$. For every $v=(m_1,\dots,m_n)\in \ZZ^n$ let $u_v:A_{n,d}\rightarrow \ZZ$ be defined by:
\begin{equation} \label{defdef}
	u_v(\ell_1\dots,\ell_n)=\prod_{i=1}^nm_i^{\ell_i}	
\end{equation}
of course $\alpha,u_v$ can be thought of as an element of $\RR^{k=\abs{A_{n,d}}}$. By direct calculation, Equation (\ref{ratEq}) implies that for every $v\in \ZZ^n$:
\begin{equation}
	\sum_{\ell\in A_{n,d}}q_{\ell}\alpha_{\ell}u_v(\ell)\in \QQ
\end{equation}
for some rational numbers $q_{\ell}$. By (\ref{defdef}) there exist $v_1,\dots,v_k\in \ZZ^n$ such that $u_{v_1},\dots,u_{v_k}$ is a basis for $\RR^k$. For example, we can take $v_i=(2^{i-1},\dots,p_m^{i-1})$ for $i=1,\dots,k$ where $p_{j}$ is the $j$'th prime. The vectors $u_{v_1},\dots,u_{v_k}$ will form a basis for $\RR^k$ since they are the columns of a Vandermonde matrix (this follows from (\ref{defdef})) and by the uniqueness of the prime decomposition, the entries of the vector $u_{v_2}$ are all different. This shows that the determinant will be non-zero. 
Therefore by the above equation every $\QQ$-linear combination of elements from the set $\{\alpha_{\ell}:\abs{\ell}=d\}$ is rational which contradicts the assumption. 
\end{proof}
Second, we prove Theorem \ref{probOp}:
\begin{proof}[Proof of Theorem \ref{probOp}]
Assume that the theorem holds for all polynomials of degree strictly less than $d$. If $\alpha_{\ell}$ is rational for every $\abs{\ell}=d$, there exists $q\in \ZZ$ such that $q\alpha_{\ell}\in \ZZ$ for any $\abs{\ell}=d$. Then the quantities $F(qv+r)\bmod 1$ for $v\in \ZZ^n$ and fixed $r\in \{0,\dots,q-1\}^n$ coincide with values of polynomials of degree strictly smaller than $d$ for which the conditions of the theorem hold. By induction we get that the values of $F$ are equidistributed modulo 1. Therefore we may assume that $\alpha_{\ell}\notin \QQ$ for some $\abs{\ell}=d$. 	Lemma \ref{a} shows that under the assumptions of Theorem \ref{WeylGen}, $F$ is related to $\mu:=\delta_{\alpha_0}$ for some irrational $\alpha_0\in (0,1)$. 
By Theorem \ref{WeylGen}, $F(\ZZ^n)\bmod1$ is equidistributed, as desired. 
\end{proof}
We conclude this note with a proof of Corollary \ref{normED}. Similarly to the proof of Theorem \ref{probOp}, we do it by finding a direction $v\in \ZZ^n$ and a measure $\mu$ such that the norm function together with $v$ and $\mu$ satisfy (\ref{homCond})-(\ref{pushCondHom}). Then we prove that $\mu$ satisfies the conditions of Theorem \ref{WeylGen} and deduce uniform distribution.
\begin{proof}[Proof of Corollary \ref{normED}]
	Fix $p\in (1,\infty)$ and let $F:\RR^n\rightarrow \RR_+$ be defined by $x\mapsto \norm{x}_{p}=\left(\sum_{i=1}^n\abs{x_i}^p\right)^{1/p}$. Since for any $i=1,\dots,n$ $F$ is invariant under the map $e_i\mapsto -e_i,e_j\mapsto e_j;\forall j\neq i$ it is sufficient to prove that $F\mid_{\RR_+^n}\bmod1(\ZZ^n)$ is equidistributed. 
	 Denote $v=e_j$ for some $j=1,\dots,n$ (arbitrarily chosen). Note that $F$ is homogeneous of degree 1 therefore $\frac{\partial}{\partial x_j}F$,$\frac{\partial^2}{\partial x_j^2}F$ are homogeneous of degree $0,-1$ respectively and so, $F$ satisfies Equations (\ref{homCond}),(\ref{pushCondHom}) for $d=1$. It remains to check that (\ref{pushCond}) holds. Indeed, for any $(x_1\dots,x_n)\in \RR^n_+$:
	\begin{equation}
		\frac{\partial}{\partial x_j}F(x_1,\dots,x_n)=x_j^{p-1}\left(\sum_{i=1}^nx_i^p\right)^{\frac{1}{p}-1}
	\end{equation}
	and denoting $S_p=\{x\in \RR_+^n:\norm{x}_p=1\}$, the function $G(x_1,\dots,x_n):=\frac{\partial}{\partial x_j}F\mid_{S_p} (x_1,\dots,x_n)=x_j^{p-1}$ satisfies $(G\bmod1)_*\mu_{S_p}\ll\lambda$ where $\mu_{S_p}$ is the surface measure on $S_p$, since the degree of the Jacobian of $F$ is 1 at every point. By Theorem \ref{WeylGen}, $F\mid_{\RR_+^n}(\ZZ^n)\bmod1$ is indeed equidistributed. 
\end{proof}


\begin{thebibliography}{0}
\bibitem{PipThes}
Harron, P.
\textit{The Equidistribution of Lattice Shapes of Rings of Integers of Cubic, Quartic, and Quintic Number Fields: an Artist’s Rendering.} PhD Thesis, Princeton University, 2016.

\bibitem{Weyl}
Weyl, H.
\textit{\"{U}ber die {G}leichverteilung von {Z}ahlen mod. {E}ins.}
Math. Ann. 77(3):313--352, 1916.

\bibitem{Bosh}
Boshernitzan, M. D.
\textit{Uniform distribution and {H}ardy fields.}
J. Anal. Math., 62:225--240, 1994.

\bibitem{PolWeyl}
Arhipov, G. I. and Karacuba, A. A. and \v{C}ubarikov, V. N.
\textit{Distribution of fractional parts of polynomials of several
              variables.}
Mat. Zametki, 25(1)3--14, 157, 1979.
\bibitem{MyFavBook}
Einsiedler, M. and Ward, Thomas
\textit{Ergodic Theory with a View towards Number Theory.}
volume 259 of Graduate Texts in Mathematics. Springer-Verlag London, Ltd., London, 2011.
\bibitem{stat}
Yifrach, Y.
\textit{The Equidistribution of Grids of Rings of Integers in Number Fields of Degrees 3,4 and 5.}
arXiv preprint, 2201.10942 ,2022.
\bibitem{pom}
Davenport, H.
\textit{On a Principle of Lipschitz} Journal of the London Mathematical Society, Volume s1-26, Issue 3, July 1951, Pages 179–183, https://doi.org/10.1112/jlms/s1-26.3.179.
\end{thebibliography}
\end{document}